\documentclass[10pt,a4paper]{article}
\usepackage{amsfonts,amsmath,amssymb,amsthm,hyperref}
\usepackage[margin=1in]{geometry}

\title{An Alternative Proof of Hesselholt's Conjecture on Galois Cohomology of Witt Vectors of Algebraic Integers}
\author{Wilson Ong}
\date{}

\newtheorem{thm}{Theorem}[section]
\newtheorem{prop}[thm]{Proposition}
\newtheorem{lem}[thm]{Lemma}

\theoremstyle{definition}

\theoremstyle{remark}

\DeclareMathOperator{\Gal}{Gal}
\DeclareMathOperator{\tr}{tr}

\begin{document}

\maketitle

\begin{abstract}
Let $K$ be a complete discrete valuation field of characteristic zero with residue field $k_K$ of characteristic $p>0$. Let $L/K$ be a finite Galois extension with Galois group $G=\Gal(L/K)$ and suppose that the induced extension of residue fields $k_L/k_K$ is separable. Let $\mathbb{W}_n(\cdot)$ denote the ring of $p$-typical Witt vectors of length $n$. Hesselholt conjectured that the pro-abelian group $\{H^1(G,\mathbb{W}_n(\mathcal{O}_L))\}_{n\geq 1}$ is isomorphic to zero. Hogadi and Pisolkar have recently provided a proof of this conjecture. In this paper, we provide an alternative proof of Hesselholt's conjecture which is simpler in several respects.
\end{abstract}

\section{Literature Review}

Let $K$ be a complete discrete valuation field of characteristic zero with residue field $k_K$ of characteristic $p>0$. Let $L/K$ be a finite Galois extension with Galois group $G=\Gal(L/K)$ and suppose that the induced extension of residue fields $k_L/k_K$ is separable. Let $\mathbb{W}_n(\cdot)$ denote the ring of $p$-typical Witt vectors of length $n$. In Hesselholt's paper \cite{hesselholt} it is conjectured that the pro-abelian group $\{H^1(G,\mathbb{W}_n(\mathcal{O}_L))\}_{n\geq 1}$ is isomorphic to zero, and the conjecture is reduced to the case where $L/K$ is a totally ramified cyclic extension of degree $p$. Let $\sigma$ be a generator of $G$ and let $t:=v_L(\sigma(\pi_L)-\pi_L)-1$ denote the ramification break (see \cite[Chapter V, \S 3]{serre}) in the ramification filtration of $G$. Recall that $t$ does not depend on the choice of generator $\sigma$.

\bigskip

Hesselholt shows his conjecture holds for extensions with $t>\frac{e_K}{p-1}$. Hogadi and Pisolkar have recently provided a proof of the conjecture for all Galois extensions (see \cite{hogadi}). In this paper, we provide an alternative proof of Hesselholt's conjecture which is simpler in several respects. First let us recall some lemmas from \cite{hesselholt}:

\begin{lem}\label{lem:trineq}
For all $a\in\mathcal{O}_L$, $v_K(\tr(a))\geq\frac{v_L(a)+t(p-1)}{p}$.
\end{lem}
\begin{proof}
We know $a\in\mathfrak{p}_L^{v_L(a)}$, so from \cite[Chapter V, \S 3, Lemma 4]{serre}, we have $\tr(a)=\pi_K^{\lfloor\frac{(t+1)(p-1)+v_L(a)}{p}\rfloor}b$
for some $b\in\mathcal{O}_K$. Now taking $K$-valuations gives the desired result.
\end{proof}

\begin{lem}\label{lem:trdiff}
For all $a\in\mathcal{O}_L$, $v_K(\tr(a^p)-\tr(a)^p)=v_K(p)+v_L(a)$.
\end{lem}
\begin{proof}
This follows by expanding $\tr(a^p)-\tr(a)^p$ using the multinomial formula and grouping the resulting expression into summands with distinct valuations. See the proof of \cite[Lemma 2.2]{hesselholt} for the details. 
\end{proof}

Next, we provide an alternative elementary proof of \cite[Lemma 2.4]{hesselholt}:

\begin{lem}\label{lem:nonzeroclass}
Suppose that $a\in\mathcal{O}_L^{\tr=0}$ represents a non-zero class in $\frac{\mathcal{O}_L^{\tr=0}}{(\sigma-1)\mathcal{O}_L}$. Then $v_L(a)\leq t-1$.
\end{lem}
\begin{proof}
For each $0\leq \mu\leq p-1$, define $x_\mu=\prod_{0\leq i<\mu}\sigma^i(\pi_L)$. It is clear $v_L(x_\mu)=\mu$. Suppose
\[
a_0x_0+a_1x_1+\cdots+a_{p-1}x_{p-1}=0
\]
for some $a_0,a_1,\ldots,a_{p-1}\in K$. The summands on the left have distinct $L$-valuations modulo $p$ and thus distinct $L$-valuations, implying each summand must be zero by the non-archimedean property. Hence the $x_\mu$ are linearly independent over $K$ and thus span $L$ over $K$. Now recall $\frac{\ker(\tr)}{(\sigma-1)L}=H^1(G,L)=0$ (see \cite[Chapter VIII, \S 4]{serre} and \cite[Chapter X, \S 1, Proposition 1]{serre}). Hence $\mathcal{O}_L^{\tr=0}\subseteq(\sigma-1)L$ so we can write
\[
a=b_1(\sigma-1)x_1+b_2(\sigma-1)x_2+\cdots+b_{p-1}(\sigma-1)x_{p-1}
\]
for some $b_1,b_2,\ldots,b_{p-1}\in K$. It is clear from the definition of $x_\mu$ that $\pi_L\sigma(x_\mu)=x_\mu\sigma^\mu(\pi_L)$ for each $1\leq \mu\leq p-1$ so that $v_L((\sigma-1)x_\mu)=v_L(\frac{(\sigma^\mu-1)\pi_L}{\pi_L}\cdot x_\mu)=t+\mu$, implying the summands on the right have distinct $L$-valuations modulo $p$, and thus distinct $L$-valuations. Since $a\not\in(\sigma-1)\mathcal{O}_L$ by hypothesis, we must have $b_{\mu '}\not\in\mathcal{O}_K$ for some $\mu '$ so that $v_L(b_{\mu '}(\sigma-1)x_{\mu '})\leq-p+t+\mu '\leq-p+t+(p-1)$ for this $\mu '$. Hence by the non-archimedean property, we conclude $v_L(a)\leq t-1$, as required.
\end{proof}

\begin{lem}\label{lem:rmzeroforn1}
Let $m\geq 1$ be an integer and suppose that the map
\[
R^m_*\colon H^1(G,\mathbb{W}_{m+n}(\mathcal{O}_L))\rightarrow H^1(G,\mathbb{W}_n(\mathcal{O}_L))
\]
is equal to zero, for $n=1$. Then the same is true for all $n\geq 1$.
\end{lem}
\begin{proof}
This follows from the long exact sequence of cohomology. See the proof of \cite[Lemma 1.1]{hesselholt} for the details.
\end{proof}

\section{Proof of Hesselholt's Conjecture}

Recall for each $n\geq 0$, we have the Witt polynomial
\[
W_n(X_0,X_1,\ldots,X_n)=X_0^{p^n}+pX_1^{p^{n-1}}+\cdots+p^nX_n=\sum_{i=0}^n p^iX_i^{p^{n-i}}
\]
Fix any $m\geq 0$. Let
\[
\sum_{i=0}^{p-1} (X_{i,0},X_{i,1},\ldots,X_{i,m})=(z_0,z_1,\ldots,z_m)
\]
where on the left we have a sum of Witt vectors. Then we know each $z_n$ is a polynomial in $\mathbb{Z}[\{X_{i,j}\}_{0\leq i\leq p-1,0\leq j\leq n}]$ with no constant term (see \cite[Chapter II, \S 6, Theorem 6]{serre}). By construction of Witt vector addition (see \cite[Chapter II, \S 6, Theorem 7]{serre}) we have
\[
\sum_{i=0}^{p-1} W_n(X_{i,0},X_{i,1},\ldots,X_{i,n})=W_n(z_0,z_1,\ldots,z_n)
\]
for each $0\leq n\leq m$. Now using the expression for the Witt polynomial $W_n$ and dividing through by $p^n$ yields
\begin{equation}\label{eq:zn}
f_n+\sum_{i=0}^{p-1} X_{i,n}-z_n=0
\end{equation}
where
\begin{equation}\label{eq:fn1}
f_n=\frac{1}{p^n}\left(\sum_{i=0}^{p-1}X_{i,0}^{p^n}-z_0^{p^n}\right)+\frac{1}{p^{n-1}}\left(\sum_{i=0}^{p-1}X_{i,1}^{p^{n-1}}-z_1^{p^{n-1}}\right)+\cdots+\frac{1}{p}\left(\sum_{i=0}^{p-1}X_{i,n-1}^p-z_{n-1}^p\right)
\end{equation}
Now for any $1\leq n\leq m$, we may add and subtract $\frac{1}{p}(-f_{n-1})^p$ to obtain
\begin{equation}\label{eq:fn2}
f_n=g_{n-2}+\frac{1}{p}\left(\sum_{i=0}^{p-1}X_{i,n-1}^p-z_{n-1}^p-(-f_{n-1})^p\right)
\end{equation}
where
\begin{equation}\label{eq:gn}
g_{n-2}=\frac{1}{p^n}\left(\sum_{i=0}^{p-1} X_{i,0}^{p^n}-z_0^{p^n}\right)+\cdots+\frac{1}{p^2}\left(\sum_{i=0}^{p-1} X_{i,n-2}^{p^2}-z_{n-2}^{p^2}\right)+\frac{1}{p}(-f_{n-1})^p
\end{equation}

\begin{lem}\label{lem:gn}
Suppose $(a_0,a_1,\ldots,a_m)\in \mathbb{W}_{m+1}(\mathcal{O}_L)$. Then
\[
v_L(g_{n-2}|_{X_{i,j}=\sigma^i(a_j)})\geq p^2\cdot\min\{v_L(a_j)\colon 0\leq j\leq n-2\}
\]
for each $2\leq n\leq m$.
\end{lem}
\begin{proof}
From (\ref{eq:zn}) and (\ref{eq:fn1}) we know $f_n$ is a polynomial in $\mathbb{Z}[\{X_{i,j}\}_{0\leq i\leq p-1,0\leq j\leq n-1}]$ with no constant term, and each monomial of $f_n$ has degree at least $p$. From (\ref{eq:zn}) we know $\sum_{i=0}^{p-1}X_{i,n-1}=z_{n-1}-f_{n-1}$, implying $\sum_{i=0}^{p-1}X_{i,n-1}^p\equiv z_{n-1}^p+(-f_{n-1})^p\pmod{p}$, so in view of (\ref{eq:fn2}) we see $g_{n-2}$ has integer coefficients. Thus from (\ref{eq:gn}) we know $g_{n-2}$ is a polynomial in $\mathbb{Z}[\{X_{i,j}\}_{0\leq i\leq p-1,0\leq j\leq n-2}]$ with no constant term, and each monomial of $g_{n-2}$ has degree at least $p^2$. Hence recalling that $v_L(\sigma^i(a_j))=v_L(a_j)$ (see \cite[Chapter II, \S 2, Corollary 3]{serre}), and using the properties of valuations, it is clear we have the desired inequality.
\end{proof}

\begin{lem}\label{lem:ineq}
Suppose $(a_0,a_1,\ldots,a_m)\in \mathbb{W}_{m+1}(\mathcal{O}_L)^{\tr=0}$. Then
\[
v_L(a_{n-1})\geq\min\{\frac{v_L(a_n)+t(p-1)}{p},t(p-1)\}
\]
for each $1\leq n\leq m$.
\end{lem}
\begin{proof}
Since $(a_0,a_1,\ldots,a_m)\in \mathbb{W}_{m+1}(\mathcal{O}_L)^{\tr=0}$, by definition of the $z_n$ we can take $z_n=0$ for $0\leq n\leq m$ and $X_{i,j}=\sigma^i(a_j)$. Then from (\ref{eq:zn}) we see $-f_n=\tr(a_n)$ for each $n$, and hence (\ref{eq:fn2}) reduces to $\frac{\tr(a_{n-1}^p)-\tr(a_{n-1})^p}{p}=-\tr(a_n)-g_{n-2}|_{X_{i,j}=\sigma^i(a_j)}$. Taking $K$-valuations of both sides of this equation then applying Lemma \ref{lem:trdiff} and Lemma \ref{lem:trineq} gives
\begin{equation}\label{eq:ineq}
v_L(a_{n-1})\geq\min\{\frac{v_L(a_n)+t(p-1)}{p},v_K(g_{n-2}|_{X_{i,j}=\sigma^i(a_j)})\}
\end{equation}
Since $f_0=0$ by (\ref{eq:fn1}), we see $g_{-1}=0$ by (\ref{eq:gn}). Hence taking $n=1$ in (\ref{eq:ineq}), we see that the claim holds for $n=1$. Now for the inductive step let $N\geq 2$ and suppose the claim holds for all $1\leq n\leq N-1$. Then we have
\begin{eqnarray*}
v_L(a_{N-1})&\geq&\min\{\frac{v_L(a_N)+t(p-1)}{p},\frac{1}{p}\cdot v_L(g_{N-2}|_{X_{i,j}=\sigma^i(a_j)})\}\\
&\geq&\min\{\frac{v_L(a_N)+t(p-1)}{p},\frac{1}{p}\cdot p^2\cdot\min\{v_L(a_{n-1})\colon 1\leq n\leq N-1\}\}\\
&\geq&\min\{\frac{v_L(a_N)+t(p-1)}{p},\frac{1}{p}\cdot p^2\cdot \frac{t(p-1)}{p}\}
\end{eqnarray*}
where the first inequality follows from (\ref{eq:ineq}), the second by Lemma \ref{lem:gn}, and the third by the induction hypothesis. This completes the inductive step.
\end{proof}

By Lemma \ref{lem:rmzeroforn1}, and recalling that $H^1(G,\mathbb{W}_{m+1}(\mathcal{O}_L))=\frac{\mathbb{W}_{m+1}(\mathcal{O}_L)^{\tr=0}}{(\sigma-1)\mathbb{W}_{m+1}(\mathcal{O}_L)}$ (see \cite[Chapter VIII, \S 4]{serre}), the following proposition (a generalisation of \cite[Proposition 2.5]{hesselholt}) proves Hesselholt's conjecture.

\begin{prop}\label{prop:rmzero}
The map
\begin{eqnarray*}
R^m_*\colon\frac{\mathbb{W}_{m+1}(\mathcal{O}_L)^{\tr=0}}{(\sigma-1)\mathbb{W}_{m+1}(\mathcal{O}_L)}&\rightarrow&\frac{\mathcal{O}_L^{\tr=0}}{(\sigma-1)\mathcal{O}_L}\\
(a_0,a_1,\ldots,a_m)&\mapsto& a_0
\end{eqnarray*}
is equal to zero, provided that $p^m>t$.
\end{prop}
\begin{proof}
Suppose $(a_0,a_1,\ldots,a_m)\in \mathbb{W}_{m+1}(\mathcal{O}_L)^{\tr=0}$. Note that $v_L(a_n)>t-p^n$ implies
\[
v_L(a_{n-1})\geq\min\{\frac{v_L(a_n)+t(p-1)}{p},t(p-1)\}>\frac{(t-p^n)+t(p-1)}{p}=t-p^{n-1}
\]
Since $v_L(a_m)>t-p^m$ by hypothesis, we see $v_L(a_0)>t-p^0$ by downward induction. Thus by Lemma \ref{lem:nonzeroclass} we see $a_0$ must represent the zero class in $\frac{\mathcal{O}_L^{\tr=0}}{(\sigma-1)\mathcal{O}_L}$.
\end{proof}

\bibliographystyle{unsrt}
\bibliography{biblio}
\end{document}